\documentclass[12pt]{article}
\usepackage{amsmath}
\usepackage{graphicx,psfrag,epsf}
\usepackage{enumerate}
\usepackage{natbib}
\usepackage{url} 
\usepackage{multicol}
\usepackage{upgreek}
\usepackage{dirtytalk}

\newcommand{\blind}{0}

\usepackage{amsfonts}
\usepackage{amsthm}
\usepackage{booktabs}

\theoremstyle{definition}
\newtheorem{theorem}{Theorem}[section]
\newtheorem{lemma}[theorem]{Lemma}
\newtheorem{corollary}[theorem]{Corollary}
\newtheorem{proposition}[theorem]{Proposition}
\newtheorem{definition}[theorem]{Definition}

\newtheorem{remark}{Remark}

\newcommand{\EE}{\mathbb{E}}
\newcommand{\VV}{\mathbb{V}}
\newcommand{\CC}{\mathbb{C}}
\allowdisplaybreaks 

\begin{document}

\def\spacingset#1{\renewcommand{\baselinestretch}%
{#1}\small\normalsize} \spacingset{1}


\if0\blind
{
  \title{\bf Moment-Type Estimators for the Dirichlet and the Multivariate Gamma Distributions}
  \author{Ioannis Oikonomidis\thanks{
    Ioannis Oikonomidis gratefully acknowledges the Bodossaki Foundation \citep{bodossaki2023} for funding his doctoral studies.}\hspace{.2cm}\\
    Department of Mathematics, National and Kapodistrian University of Athens\\
    and \\
    Samis Trevezas\\
    Department of Mathematics, National and Kapodistrian  University of Athens}
  \maketitle
} \fi

\if1\blind
{
  \bigskip
  \bigskip
  \bigskip
  \begin{center}
    {\LARGE\bf Moment-Type Estimators for the Dirichlet and the Multivariate Gamma Distributions}
\end{center}
  \medskip
} \fi

\bigskip
\begin{abstract}
This study presents new closed-form estimators for the Dirichlet and the Multivariate Gamma distribution families, whose maximum likelihood estimator cannot be explicitly derived. The methodology builds upon the score-adjusted estimators for the Beta and Gamma distributions, extending their applicability to the Dirichlet and Multivariate Gamma distributions. Expressions for the asymptotic variance-covariance matrices are provided, demonstrating the superior performance of score-adjusted estimators over the traditional moment ones. Leveraging well-established connections between Dirichlet and Multivariate Gamma distributions, a novel class of estimators for the latter is introduced, referred to as \say{Dirichlet-based moment-type estimators}. The general asymptotic variance-covariance matrix form for this estimator class is derived. To facilitate the application of these innovative estimators, an R package called \say{estimators} is developed and made publicly available.
\end{abstract}
\noindent
{\it Keywords:} Dirichlet, Multivariate Gamma, Score-adjusted estimators, Closed-form estimators, Maximum-likelihood estimation, Method of moments
\vfill
\newpage
\spacingset{1.45} 
\section{Introduction}
\label{sec:intro}
\subsection{Study Domain}
\label{ss:study-domain}
This study concerns the foundational framework in parametric statistics that deals with independent and identically distributed (iid) observations, presumed to be drawn from a distribution family with some unknown parameters requiring estimation. Within this context, two prominent techniques have emerged: the method of moments (ME), introduced by \cite{pearson1894contributions}, and the method of maximum likelihood (MLE), developed by \cite{fisher1922mathematical}. While both methods are now widely accepted in the scientific community, it's important to note that during their inception, a vigorous dispute surrounded them, accompanied by language that did not align with the norms of scientific discourse. Readers interested in the history of statistics are encouraged to delve deeper into the subject, starting with \say{The Epic Story of Maximum Likelihood} by \cite{stigler2007epic}. \par
Even though this debate can be characterized as highly productive in uncovering each method's weaknesses, an argument could be made that it also obscured the similarities between the two. In exponential families, the MLE is, in fact, a moment estimator applied to the minimal sufficient statistics \citep{davidson1974moment}. Under this perspective, the pivotal question is not whether moments should play a role in estimation, as Karl Pearson mockingly asked Ronald Fisher, \say{Wasting your time fitting curves by moments, eh?} \citep{pearson1936method}, but rather \textit{which} specific moments should be utilized. \par
While the asymptotic efficiency of the MLE is undoubtedly alluring, its applicability encounters challenges in situations where the estimator cannot be directly determined; this is indeed the case for common distribution families such as Beta and Gamma. In these instances, iterative algorithms have to be utilized to derive the MLE numerically. This situation prompts a delicate balance between the estimator's efficiency and its computational cost. In application settings, it is not unreasonable to settle for a theoretically inferior estimation method in order to minimize this cost. Even when time is not an issue and an iterative scheme is employed to calculate the MLE, poor initialization may cause a failure of convergence and jeopardize the estimation process. Recently, a search for alternative explicit estimators that display properties superior to the ME has been gaining attention.
\subsection{State of the Art}
\label{ss:state-of-the-art}
The quest for closed-form estimators sparked in 2017 with a paper by \cite{ye2017closed} on the Gamma distribution, $X\sim\mathcal{G}(\alpha, \beta)$. In this work, the authors utilized the Generalized Gamma distribution, $X^{\gamma}\sim\mathcal{GG}(\alpha, \beta, \gamma)$, adding a new parameter $\gamma$ only to create a third score equation \citep{stacy1962generalization}. By setting $\gamma$ back to $1$, they ended up with a 3-equation system and 2 unknown parameters, allowing them to discard one equation and explicitly solve the system with respect to $(\alpha, \beta)$. The authors showed that these closed-form estimators are consistent and asymptotically normal, with a higher relative efficiency than the classic ME estimators. This methodology served as the foundation for two distinct research directions. \par
The first extension concerns employing generalized versions of the distribution under consideration and keeping a subset of the score equations that can be explicitly solved. This idea is best illustrated by \cite{kim2021new}, who observed that the power transformation performed by \cite{ye2017closed} can serve as a unifying framework encompassing multiple distributions (Gamma, Nakagami, Wilson–Hilferty, inverse Gamma, and others) simply by setting $\gamma$ to different values. The authors extended this methodology to other distributions by employing a generalized Box-Cox transformation, and furthermore derived bias-corrected forms of the estimators. Notable instances of this methodology include \cite{zhao2021closed} for the Nakagami distribution, \cite{kim2021lindley} for the weighted Lindley distribution, and \cite{chen2022novel} for the Beta distribution. \par
The second extension is best described in \cite{tamae2020score}, which highlighted that the Generalized Gamma distribution or any other transformation is not necessary, since the estimators of \cite{ye2017closed} are in fact moment-type estimators that can be derived by equating $\EE(X) = \alpha\beta$ and $\CC(X, \log X) = \beta$ to their sample counterparts. The authors called these the score-adjusted moment estimators (SAME) and applied this methodology to derive closed-form estimators for the parameters of the Beta, Gamma-Poisson, and Beta-Binomial distributions \citep{tamae2020score, tamae2022score}. However, it's important to note that in the case of the Gamma distribution, these estimators were originally derived by \cite{wiens2003class}, who used the same methodology as \cite{tamae2020score}, except for the fact that the sample covariance set equal to $\CC(X, \log X)$ was the unbiased one (dividing by $n-1$ instead of $n$). \par 
\cite{Louzada2019} derived bias-corrected versions of the Ye–Chen estimators. It is noteworthy that \cite{papadatos2022point} reached independently the same estimators for Beta and Gamma through the application of a Stein-type covariance identity and U-statistics. \par
Considerable work has been done to extend the methodology to the multivariate Gamma distribution $\mathcal{MG}_k(\boldsymbol{\alpha}, \beta)$, where $\boldsymbol{\alpha}$ is a k-dimensional vector \citep{mathal1992form}. Once again, both methodological approaches are encountered, by two different research teams. \cite{zhao2022closed} directly extended the original idea of \cite{ye2017closed} to the bivariate case, utilizing the Generalized Bivariate Gamma $\mathcal{GBG}(\alpha_1, \alpha_2, \beta, \gamma_1, \gamma_2)$ distribution and setting $\gamma_1 = \gamma_2 = 1$. The computations involved cannot be easily generalized in the k-dimensional case, therefore the authors proposed a one-step MLE instead, which achieves asymptotic efficiency \citep{jang2023new}. \par 
\cite{nawa2023new} also worked on the Bivariate Gamma $\mathcal{BG}(\alpha_1, \alpha_2, \beta)$, following the score-adjusted approach of \cite{tamae2020score}. Specifically, the authors utilized a well-known relation between Beta and Gamma distributions to generate a sample from $\mathcal{B}(\alpha_1, \alpha_2)$ and yield the Beta SAME estimators for $(\alpha_1, \alpha_2)$, which in turn allowed them to estimate the $\beta$ parameter. The authors highlight that this methodology could be used for the multivariate Beta (Dirichlet) as well. However, since score-adjusted estimators for the Dirichlet distribution have not yet been published, \cite{nawa2023closed} relied on the moment estimators instead for the general k-dimensional case.
\subsection{Study Innovation}
\label{ss:study-innovation}
This study focuses on the Dirichlet and Multivariate Gamma distributions, deriving new explicit estimators. A methodological foundation in the context of exponential families is established, moving the focus from individual estimator properties to general classes. \par
The score-adjusted estimators proposed by \cite{tamae2020score} for the Beta and Gamma distributions are extended to Dirichlet and Multivariate Gamma. Expressions of the asymptotic variance-covariance matrices for the ME, SAME, and MLE are provided, showing the superiority of the score-adjusted estimators in comparison to the classic moment ones. Deriving the SAME for the Dirichlet distribution removes the obstacle hindering \cite{nawa2023closed} in the Multivariate Gamma case. Using well-known properties connecting the Dirichlet and the Multivariate Gamma distributions, a new class of estimators for the Multivariate Gamma is proposed, called \textit{Dirichlet-based moment-type estimators}. The asymptotic variance-covariance matrix of the class members is derived. \par
In order to allow for the application of the results presented, an R package named \texttt{estimators} is published on CRAN, implementing the ME, MLE, and SAME estimators for the Beta, Gamma, Dirichlet, and Multivariate Gamma distributions (See Supplementary Material).
\section{Foundations}
\label{sec:foundations}
This section provides a brief overview of key results in estimation methods and asymptotic statistics. It serves as a unifying framework, demonstrating the generality of these results across moment-type estimators, thus simplifying individual proofs of consistency and asymptotic normality. Two great textbooks on mathematical statistics covering these results are \cite{lehmann2022testing} and \cite{van2000asymptotic}. \par 
For conciseness of the text, the parameter vector $\boldsymbol{\theta}$ will be omitted from the moment notation, thus writing $\EE(\mathbf{X})$ instead of $\EE_{\boldsymbol{\theta}}(\mathbf{X})$. To highlight the fact that expectation and variance are functions of the parameters, the notation $e_{h}(\boldsymbol{\theta}):=\EE_{\boldsymbol{\theta}}(h(\mathbf{X}))$ and $v_{h}(\boldsymbol{\theta}):=\VV_{\boldsymbol{\theta}}(h(\mathbf{X}))$ will be used whenever necessary. Two common abuses of notation are adopted in this text: Using $g(\mathbf{x})$ with a univariate function $g$ to denote the element-wise application $(g(x_1), \dots, g(x_k))$, and concatenating vectors $(\mathbf{x}, \mathbf{y})$ to denote $(x_1, \dots, x_k, y_1, \dots, y_l)$. Finally, in Sections \ref{s:dirichlet} and \ref{s:multivariate-gamma} where the estimators for the Dirichlet and the Multivariate Gamma distributions are respectively presented, the index $n$ indicating the sample size will be omitted for simplicity, thus writing $\hat{\boldsymbol{\theta}}$ instead of $\hat{\boldsymbol{\theta}}_n$.
\subsection{Parameter Estimation}
\begin{definition}\label{def:me}
    Let $\mathbf{X}_1, \dots, \mathbf{X}_n$ be a random sample from a parametric distribution family $\mathcal{D}_{\boldsymbol{\theta}}$ with support $\mathbf{S}\subset\mathbb{R}^k$ and parameter space $\boldsymbol{\Theta}\subset\mathbb{R}^d$. Let $h:\mathbf{S}\rightarrow\mathbb{R}^m$ be a function such that $e_h: \boldsymbol{\Theta} \rightarrow\mathbb{R}^m$ is an injection, and $g:\mathbb{R}^m\rightarrow \mathbb{R}^d$ be a function such that $g|_{e_h(\boldsymbol{\Theta})} = e_h^{-1}$. Then, the estimator
    \[
        \tilde{\boldsymbol{\theta}}_n = g\left(\overline{h\left(\mathbf{X}_n\right)}\right),
    \]
    is called a \textit{moment-type estimator} of $\boldsymbol{\theta}$. The estimator exists as long as $\tilde{\boldsymbol{\theta}}_n\in \boldsymbol{\Theta}$.
\end{definition}
Definition \ref{def:me} is fairly general. The reason $g$ is not directly defined to be the inverse of $e_h$ is to allow estimators for which $\overline{h\left(\mathbf{X}_n\right)} \notin e_h(\boldsymbol{\Theta})$, but $\tilde{\boldsymbol{\theta}}_n\in \boldsymbol{\Theta}$. A trivial example illustrating this with $\boldsymbol{\Theta} = \mathbb{R}$ would be to assume $e_h(\theta) = (\theta, \theta)$, so that $e_h(\boldsymbol{\Theta})$ is only a line on $\mathbb{R}^2$, and  a sample such that $\overline{h\left(\mathbf{X}_n\right)} = (y_1, y_2), \, y_1\neq y_2$. In this case the functions $g(y_1, y_2) = y_1$, or $g(y_1, y_2) = (y_1 + y_2) / 2$ satisfy $g|_{e_h(\boldsymbol{\Theta})} = e_h^{-1}$, and while $e^{-1}_h(y_1, y_2)$ cannot be defined, $g(y_1, y_2)$ can. \par
\begin{proposition} \label{prop:me-asymp}
    If $g$ is continuously differentiable at $\boldsymbol{\mu}:=e_h(\boldsymbol{\theta})$ and $\EE\left(||h(\mathbf{X})||^2\right)$ is finite, then the moment-type estimator $\tilde{\boldsymbol{\theta}}_n$ exists with probability tending to 1, and
    \[
        \sqrt{n}\left(\tilde{\boldsymbol{\theta}}_n-\boldsymbol{\theta}\right) \overset{\mathcal{L}}{\longrightarrow} \mathcal{N}_d\left(\mathbf{0}, \mathbf{G}\mathbf{V}\mathbf{G}^\top\right),
    \]
    where $\mathbf{G} :=\nabla g(\boldsymbol{\mu})$ and $\mathbf{V} := v_h(\boldsymbol{\theta})$.
\end{proposition}
Proposition \ref{prop:me-asymp} is an application of the delta method with the function $g$ on the Central Limit Theorem for $\boldsymbol{\mu}$. Therefore, the asymptotic variance-covariance matrix of a moment-type estimator is determined by the matrices $\mathbf{G}$ and $\mathbf{V}$.\par 
\subsection{Exponential Families}
\label{ss:exponential}
\begin{definition}\label{def:exp-fam}
    Let $\mu$ be a $\sigma$-finite measure over the measurable space $\left(\mathbb{R}^k, \mathcal{B}\left(\mathbb{R}^k\right)\right)$, where $\mathcal{B}\left(\mathbb{R}^k\right)$ is the Borel $\sigma$-algebra on $\mathbb{R}^k$, and let $\boldsymbol{\Theta}\subset\mathbb{R}^d$ be the parameter space. Given the mappings $T:\mathbb{R}^k\rightarrow \mathbb{R}^d$ and $\eta:\boldsymbol{\Theta}\rightarrow \mathbb{R}^d$, the \textit{log-partition} $A:\boldsymbol{\Theta}\rightarrow \mathbb{R}$ is defined as 
    \[
        A(\boldsymbol{\theta}):= \log\int_{\mathbb{R}^k} \exp\left\{ \eta(\boldsymbol{\theta})^\top T(\mathbf{x})\right\} \mu(\text{d}\mathbf{x}).
    \]
    A set of distributions is said to be a $k$-variate $d$-dimensional \emph{exponential family of distributions} if every distribution in the set has a probability measure of the form
    \begin{equation*}\label{exp_fam}
        \text{d}P(\mathbf{x};\boldsymbol{\theta}) = \exp\left\{\eta(\boldsymbol{\theta})^\top T(\mathbf{x}) -A(\boldsymbol{\theta})\right\} \text{d}\mu(\mathbf{x}), \quad \boldsymbol{\theta}\in\boldsymbol{\Theta}.
    \end{equation*}
    The support $\mathbf{S}\subset \mathbb{R}^k$ is independent of $\boldsymbol{\theta}$, absorbed inside $\mu(\mathbf{x})$.
\end{definition}
In exponential families, the density function can be expressed in many forms. If $\eta$ is the identity function, the family is said to be in \textit{canonical form}. In this form, $\boldsymbol{\theta}$ is called the \textit{natural parameter}. An exponential family is called \textit{regular} if $\boldsymbol{\Theta}$ is an open set. It is also called to be of \textit{full rank} if no linear combination $\eta(\boldsymbol{\theta})^\top T(\mathbf{x})$ is constant with probability 1. This study concerns regular exponential families of full rank. Proposition \ref{prop:exp-fam-res} summarizes some important results.
\begin{proposition}\label{prop:exp-fam-res}
    Let $\mathbf{X}$ be a random vector whose distribution belongs in a $d$-dimensional exponential family of distributions, expressed in its canonical form. The following results hold:
    \begin{enumerate}[(i)]
        \item The natural parameter space $\boldsymbol{\Theta}$ is convex.
        \item The function $A$ is analytic on the set $\left\{\boldsymbol{\theta}\in\mathbb{C}^d:\text{Re}\left(\boldsymbol{\theta}\right)\in\text{int}\left({\boldsymbol{\Theta}}\right)\right\}$. Its derivatives can be found by differentiating under the integral sign.
        \item The function $A$ is convex on $\text{int}\left({\boldsymbol{\Theta}}\right)$. The log-likelihood function is concave on $\text{int}\left({\boldsymbol{\Theta}}\right)$.
        \item The moment-generating function of $T(\mathbf{X})$ is $M_T(\mathbf{u}) = \exp\left\{A(\boldsymbol{\theta} + \mathbf{u}) - A(\boldsymbol{\theta})\right\}$.
        \item The following expressions hold: $e_T(\boldsymbol{\theta}) = \nabla A(\boldsymbol{\theta})$, $v_T(\boldsymbol{\theta}) = \nabla^2 A(\boldsymbol{\theta})$.
    \end{enumerate}
\end{proposition}
\begin{corollary}\label{corollary:mle}
Let $\mathbf{X}_1, \dots, \mathbf{X}_n$ be a random sample from an exponential family of distributions $\mathcal{D}_{\boldsymbol{\theta}}$ expressed in canonical form. Then, the MLE $\hat{\boldsymbol{\theta}}_n$ is a moment-type estimator, it exists with probability tending to 1, and
    \[
        \sqrt{n}\left(\hat{\boldsymbol{\theta}}_n-\boldsymbol{\theta}\right) \overset{\mathcal{L}}{\longrightarrow} \mathcal{N}_d\left(\mathbf{0}, \VV(T(\mathbf{X}))^{-1}\right).
    \]
If the estimator exists, it is the unique solution of the equation system $e_T(\boldsymbol{\theta}) = \overline{T\left(\mathbf{X}_n\right)}$.
\end{corollary}
\begin{proof}
    In exponential families, the MLE can be expressed as a moment-type estimator of the sufficient statistics \citep{davidson1974moment}, therefore it satisfies Definition \ref{def:me} with $h = T$ and $g = e^{-1}_T = (\nabla A)^{-1}$. It is straightforward to see that $\mathbf{G} = \VV(T(\mathbf{X}))^{-1} = \mathbf{V}^{-1}$ due to Proposition \ref{prop:exp-fam-res}-(5), therefore the asymptotic variance-covariance matrix $\mathbf{G}\mathbf{V}\mathbf{G}^\top$ reduces to $\VV(T(\mathbf{X}))^{-1}$, the inverse of the Fisher Information matrix. Since $g$ is a bijection, the equation system $e_T(\boldsymbol{\theta}) = \overline{T\left(\mathbf{X}_n\right)}$ either has a unique solution or is contradictory, depending on whether $\overline{T\left(\mathbf{X}_n\right)} \in e_T(\boldsymbol{\Theta})$ or not.
\end{proof}
\subsection{Polygamma Functions}
The Dirichlet and Multivariate Gamma distributions have a non-explicit normalization constant involving the polygamma functions.
\begin{definition}
    The gamma function $\Gamma(\alpha): \mathbb{R}_{+}\longrightarrow \mathbb{R}_{+}$ is defined as
    \[
    \Gamma(\alpha):=\int_{0}^{+\infty}x^{\alpha-1}e^{-x}dx.
    \]
    The digamma and polygamma of order $m$ functions are defined as
    \[
        \psi(\alpha) := \frac{d}{da}\log \Gamma(\alpha), \quad
        \psi_{m}(\alpha) := \frac{d^{m}}{da^{m}}\psi(\alpha), \quad m \in \mathbb{N}.
    \]
    The digamma and polygamma of order $m$ difference functions are defined as
    \[
        \Psi(\alpha_1, \alpha_2) := \psi(\alpha_1)-\psi(\alpha_2), \quad
        \Psi_m(\alpha_1, \alpha_2) := \psi_m(\alpha_1)-\psi_m(\alpha_2).
    \]
\end{definition}
The recurrence formula $\psi(\alpha+1) = \psi(\alpha) + 1 / \alpha$ satisfied by the digamma function generates several useful properties for the moment calculations required in this work, which are gathered in the Appendix.
\section{Dirichlet Distribution}
\label{s:dirichlet}
\subsection{Distribution Overview}
\label{ss:dirichlet-summary}
\begin{definition}
    A random vector $\mathbf{X}$ is said to follow the Dirichlet distribution with parameters $\boldsymbol{\alpha} \in (0, +\infty)^k$, denoted by $\mathbf{X}\sim\mathcal{D}_k(\boldsymbol{\alpha})$, if it has a probability density function with respect to the Lebesgue measure on $\mathbb{R}^{k-1}$ of the form
    \begin{equation*}
        f(\mathbf{x};\boldsymbol{\alpha}) = \frac{\Gamma(\alpha_0)}{\prod_{i=1}^k\Gamma(\alpha_i)}\,\prod_{i=1}^k x_i^{\alpha_i-1} \, \, \mathbb{I}_{\mathbf{S}}(\mathbf{x}),
    \end{equation*}
    where $\alpha_0 = \sum_{i=1}^k \alpha_i$ and support the open standard $(k-1)$-dimensional simplex $\mathbf{S} = \{\mathbf{x} \in (0,1)^k: \sum_{i=1}^k x_i = 1\}$.
\end{definition}
\begin{lemma}\label{dirichlet-exp-fam}
The $\mathcal{D}_k(\boldsymbol{\alpha})$ distribution is a $k$-variate $k$-dimensional exponential family in its canonical form with
\begin{equation*}
    T(\mathbf{X}) = \log \mathbf{X}, \quad
    A(\boldsymbol{\alpha}) = \sum_{i=1}^k\log \Gamma(\alpha_i) -\log \Gamma(\alpha_0).
\end{equation*}
\end{lemma}
The necessary moments to derive the estimators of interest are presented in Lemma \ref{lemma:dirichlet-moments}. A richer version, containing the moments required to calculate the estimators' asymptotic variance-covariance matrices, can be found in the Appendix (Lemmas \ref{appendix-lemma:dirichlet-moments-E} and \ref{appendix-lemma:dirichlet-moments-V}).
\begin{lemma}\label{lemma:dirichlet-moments}
Let $\mathbf{X}\sim\mathcal{D}_k(\boldsymbol{\alpha})$. Then, for $i \in [k]$:
\vspace{-0.7cm}
\begin{equation*}
    \begin{minipage}{.40\linewidth}
      \centering
        \begin{align*}
            \EE(X_i) &= \frac{\alpha_i}{\alpha_0}, \\
            \EE(\log X_i) &= \Psi(\alpha_i, \alpha_0),
        \end{align*}
    \end{minipage}
    \begin{minipage}{.60\linewidth}
      \centering 
        \begin{align*}
            \VV(X_i) &= \frac{\alpha_i(\alpha_0-\alpha_i)}{\alpha_0^2(\alpha_0+1)}, \\
            \CC(X_i, \log X_i) &= \frac{\alpha_0-\alpha_i}{\alpha_0^2}.
        \end{align*}
    \end{minipage}
\end{equation*}
\end{lemma}
\subsection{Estimators and Properties}
\label{ss:dirichlet-estimators}
In this subsection, the formulas of the Dirichlet estimators along with expressions of their asymptotic variance-covariance matrices are derived.  All proofs follow the same scheme: Specify functions $h$, $g$ showing that Definition \ref{def:me} is satisfied, calculate matrices $\mathbf{G}$ and $\mathbf{V}$ and invoke Proposition \ref{prop:me-asymp}. The form of $\mathbf{V}$ in most cases cannot be concisely expressed, therefore its variance and covariance components are moved to the Appendix Lemma \ref{appendix-lemma:dirichlet-moments-V}.
\begin{proposition} \label{prop:dirichlet-mle}
    Let $\mathbf{X}_n = (X_{1n}, \dots, X_{kn}), \, n \in [N]$ be a random sample from $\mathcal{D}_k(\boldsymbol{\alpha})$. The MLE of $\boldsymbol{\alpha}$, derived as the unique solution of the equation system
    \[
        \overline{\log(X_i)} = \Psi(\hat{\alpha}_i, \hat{\alpha}_0), \qquad i \in [k],
    \]
    is a moment-type estimator with asymptotic variance-covariance matrix $\VV(\log\mathbf{X})^{-1}$.
\end{proposition}
\begin{proof} 
    The result is immediate by Corollary \ref{corollary:mle}. The matrix $\VV(\log\mathbf{X})^{-1}$ can be retrieved from Lemma \ref{appendix-lemma:dirichlet-moments-V}.
\end{proof}
\begin{proposition} \label{prop:dirichlet-me}
    Let $\mathbf{X}_n = (X_{1n}, \dots, X_{kn}), \, n \in [N]$ be a random sample from $\mathcal{D}_k(\boldsymbol{\alpha})$. The Moment Estimator
    \[
        \tilde{\alpha}_i = \frac{\overline{X_i} (\overline{X_i} - \overline{X_i^2})}{\overline{X_i^2}-\overline{X_i}^2}, \qquad i \in [k],    
    \]
    is a moment-type estimator with functions
    \begin{align*}
        h(\mathbf{x}) &= (x_1, \dots, x_k, x^2_1, \dots, x^2_k), \\
        g_i(\mathbf{y}) &= \frac{y_i(y_i-y_{k+i})}{y_{k+i}-y^2_i}, \quad i \in [k].
    \end{align*}    
    The asymptotic variance-covariance matrix takes the form $\mathbf{G}\mathbf{V}\mathbf{G}^\top$, where $\mathbf{V} = v_h(\boldsymbol{\alpha}, \beta)$ can be retrieved from Lemma \ref{appendix-lemma:dirichlet-moments-V} and $\mathbf{G}$ is a $k\times 2k$ block diagonal matrix:
    \[
        \mathbf{G} = 
        \left[\begin{matrix}
            \mathbf{G}_{1} & \mathbf{G}_{2}
        \end{matrix}\right],
    \]
     such that, for $i \in [k]$:
    \[
        \mathbf{G}_{1ii} = \frac{\alpha_0}{\alpha_0-\alpha_i} (2\alpha_0+1)(\alpha_i+1), \quad
        \mathbf{G}_{2ii} = -\frac{\alpha_0}{\alpha_0-\alpha_i}(\alpha_0+1)^2.
    \] 
\end{proposition}
\begin{proof} 
    It is immediate to see that Definition \ref{def:me} is satisfied by functions $h$ and $g$, therefore $\tilde{\boldsymbol{\alpha}}$ is indeed a moment-type estimator. The $k\times 2k$ Jacobian matrix of $g$ can be broken down into two $k\times k$ matrices $\mathbf{G}_{1}, \mathbf{G}_{2}$, such that, for $i \in [k]$:
    \begin{align*}
        \frac{\partial g_i}{\partial y_i}(\mathbf{y}) &= \frac{y_{k+i}\left(-y^2_i+2y_i-y_{k+i}\right)}{\left(y_{k+i}-y^2_i\right)^2}, \\
        \frac{\partial g_i}{\partial y_{k+i}}(\mathbf{y}) &= \frac{y^2_i(y_i-1)}{\left(y_{k+i}-y^2_i\right)^2}.
    \end{align*}
    All other derivatives are trivially equal to $0$, therefore $\mathbf{G}_1, \mathbf{G}_2$ are diagonal. By evaluating the derivatives in $\boldsymbol{\mu} = e_h(\boldsymbol{\alpha}, \beta)$ and retrieving the required moments from Lemma \ref{appendix-lemma:dirichlet-moments-E} the form of $\mathbf{G}$ follows. 
\end{proof}
\begin{theorem} \label{th:dirichlet-same}
    Let $\mathbf{X}_n = (X_{1n}, \dots, X_{kn}), n \in [N]$ be a random sample from $\mathcal{D}_k(\boldsymbol{\alpha})$. The Score-Adjusted Moment Estimator
    \[
        \breve{\alpha}_i = \frac{(k-1)\overline{X_i}}{\sum_{j=1}^k\left[\,\overline{X_j \log X_j}-\overline{X_j} \, \overline{\log X_j}\,\right]}, \quad i \in [k],
    \]
    is a moment-type estimator with functions
    \begin{align*}
        h(\mathbf{x}) &= (x_1, \dots, x_k, \log x_1, \dots, \log x_k, x_1 \log x_1, \dots, x_k \log x_k), \\
        g_i(\mathbf{y}) &= \frac{(k-1)y_i}{\sum_{j=1}^k(y_{2k+j} - y_{j}y_{k+j})}, \quad i \in [k].
    \end{align*}
    The asymptotic variance-covariance matrix takes the form $\mathbf{G}\mathbf{V}\mathbf{G}^\top$, where $\mathbf{V} = v_h(\boldsymbol{\alpha}, \beta)$ can be retrieved from Lemma \ref{appendix-lemma:dirichlet-moments-V} and $\mathbf{G}$ is a $k\times 3k$ block matrix:
    \[
        \mathbf{G} = 
        \left[\begin{matrix}
            \mathbf{G}_{1} & \mathbf{G}_{2} & \mathbf{G}_{3}
        \end{matrix}\right],
    \]
    such that,  for $i, j \in [k]$:
    \[
        \mathbf{G_1}_{ij} = \frac{\alpha_0\alpha_i}{k-1}\Psi(\alpha_j, \alpha_0) + \alpha_0 \mathbb{I}_{\{i=j\}}, \quad
        \mathbf{G_2}_{ij} = \frac{\alpha_0\alpha_i\alpha_j}{k-1}, \quad
        \mathbf{G_3}_{ij} = -\frac{\alpha_0\alpha_i}{k-1}.
    \]
\end{theorem}
\begin{proof} 
    It is immediate to see that Definition \ref{def:me} is satisfied by functions $h$ and $g$, therefore $\breve{\boldsymbol{\alpha}}$ is indeed a moment-type estimator. The $k\times 3k$ Jacobian matrix of $g$ can be broken down into three $k\times k$ blocks, such that for $i, j \in [k]$:
    \begin{align*}
        \frac{\partial g_i}{\partial y_j}(\mathbf{y}) &= \frac{k-1}{c^2(\mathbf{y})}y_iy_{k+j} + \frac{k-1}{c(\mathbf{y})} \mathbb{I}_{\{i=j\}}, \\
        \frac{\partial g_i}{\partial y_{k+j}}(\mathbf{y}) &= \frac{k-1}{c^2(\mathbf{y})}y_iy_j, \\
        \frac{\partial g_i}{\partial y_{2k+j}}(\mathbf{y}) &= -\frac{k-1}{c^2(\mathbf{y})}y_i, 
    \end{align*}
    where $c(\mathbf{y}) := \sum_{j=1}^k(y_{2k+j}-y_jy_{k+j})$. By evaluating the derivatives in $\boldsymbol{\mu} = e_h(\boldsymbol{\alpha})$ and retrieving the required moments from Lemma \ref{appendix-lemma:dirichlet-moments-E} the form of $\mathbf{G}$ follows.
\end{proof}
\subsection{Simulation Study}
\label{ss:dirichlet-study}
\begin{figure}[ht]
    \centering
    \includegraphics[width=\textwidth]{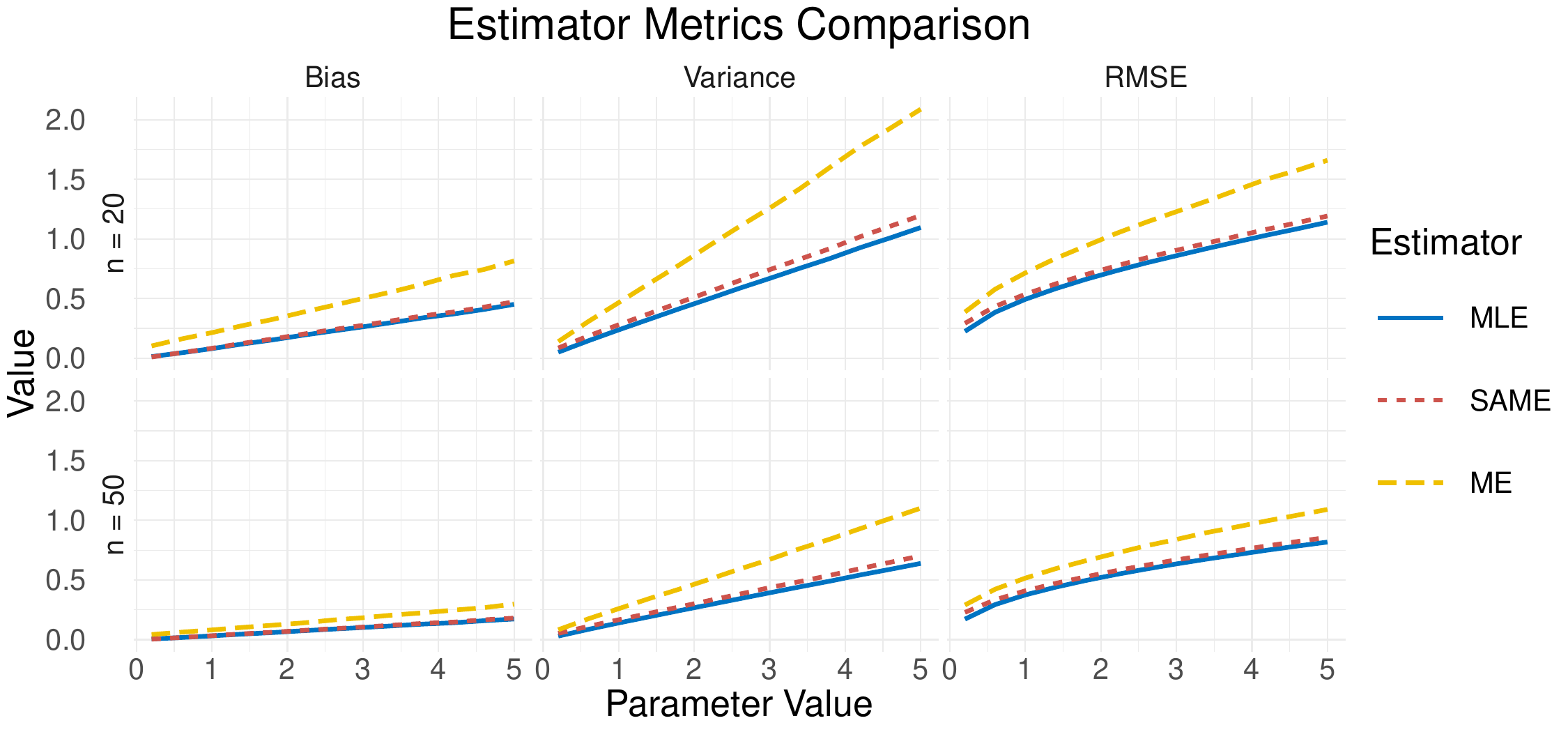}
    \caption{Bias, variance, and RMSE of the $\alpha_1$ parameter estimators in a 5-dimensional Dirichlet distribution with $\boldsymbol{\alpha} = (\alpha_1, 0.2, 1, 2, 5)$. The metrics were estimated for $n = 20$ (upper row) and $n = 50$ (lower row) observations with a Monte Carlo sample size $m = 10^5$.}
    \label{fig:dirichlet-metrics}
\end{figure}
This subsection presents the estimator comparison. For finite samples, bias, variance, and Root Mean Square Error (RMSE) are estimated using Monte Carlo simulations. Specifically, $m = 10^5$ samples of size $n = 20$ and $n = 50$ were generated from a 5-dimensional Dirichlet distribution with $\boldsymbol{\alpha} = (\alpha_1, 0.2, 1, 2, 5)$ and $\alpha_1\in [0.2, 5]$. The results are presented in Figure \ref{fig:dirichlet-metrics}. It is evident that the SAME outperforms the ME and is close to the MLE. \par
\begin{figure}[ht]
    \centering
    \includegraphics[width=\textwidth]{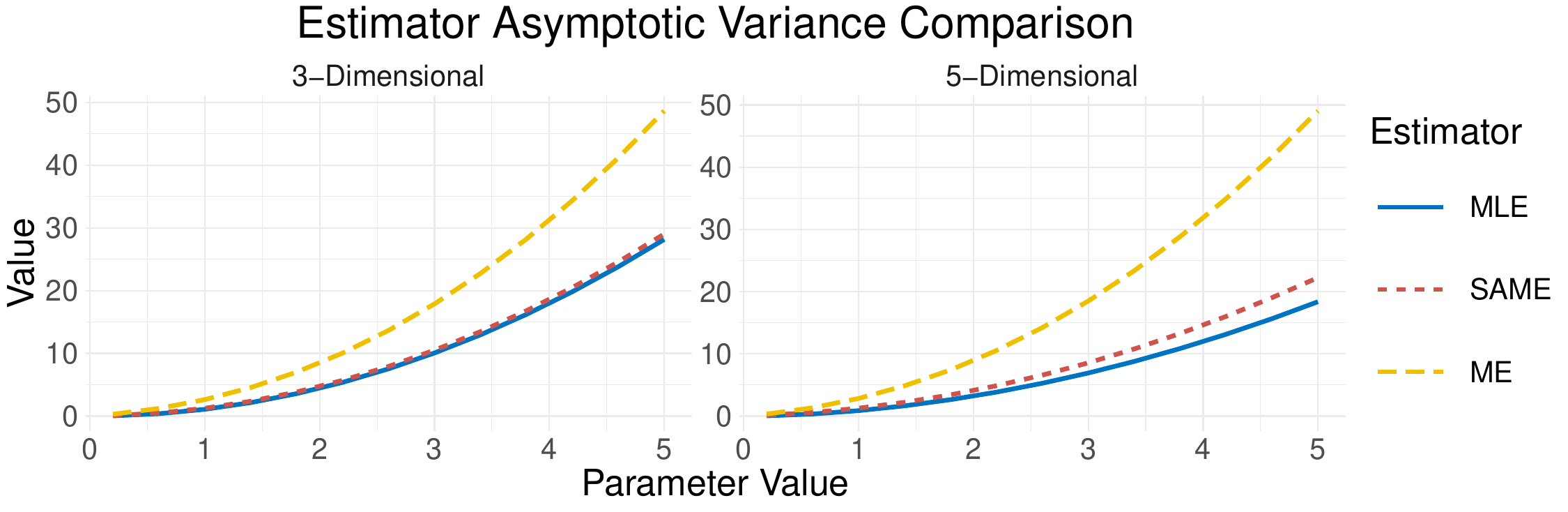}
    \caption{Asymptotic variance of the $\alpha_1\in [0.2, 5]$ parameter estimators in a 3-dimensional Dirichlet distribution with $\boldsymbol{\alpha} = (\alpha_1, 1, 5)$ and a 5-dimensional Dirichlet distribution with $\boldsymbol{\alpha} = (\alpha_1, 0.2, 1, 2, 5)$.}
    \label{fig:dirichlet-acov}
\end{figure}
The asymptotic efficiency of the estimators was examined in a 3-dimensional Dirichlet distribution with $\boldsymbol{\alpha} = (\alpha_1, 1, 5)$ and a 5-dimensional Dirichlet distribution with $\boldsymbol{\alpha} = (\alpha_1, 0.2, 1, 2, 5)$, $\alpha_1\in [0.2, 5]$. The results are presented in Figure \ref{fig:dirichlet-acov}. The outtake of this comparison is that the SAME and MLE utilize information provided from all $k$ coordinates, which results in a considerable variance decrease, in contrast to the ME which in this example is only slightly affected due to the change in $\tilde{\alpha}_0$. This is in fact a significant advantage of the SAME, making the estimator more appealing as the dimension increases.
\section{Multivariate Gamma Distribution}
\label{s:multivariate-gamma}
\subsection{Distribution Overview}
\label{ss:mgamma-summary}
\begin{definition}
    A random vector $\mathbf{X}$ is said to follow the Multivariate Gamma distribution with parameters $\boldsymbol{\alpha} \in (0, +\infty)^k$, $\beta\in (0, +\infty)$, denoted by $\mathbf{X}\sim\mathcal{MG}_k(\boldsymbol{\alpha}, \beta)$, if it has a probability density function with respect to the Lebesgue measure on $\mathbb{R}^{k}$ of the form
    \begin{equation*}
        f(\mathbf{x};\boldsymbol{\alpha}, \beta) = \frac{1}{\beta^{\alpha_0}\prod_{i=1}^k\Gamma(\alpha_i)}\, e^{-x_k/\beta}x_1^{\alpha_1-1}\prod_{i=2}^k (x_i-x_{i-1})^{\alpha_i-1} \, \, \mathbb{I}_{\mathbf{S}}(\mathbf{x}),
    \end{equation*}
    where $\alpha_0 = \sum_{i=1}^k \alpha_i$ and support $\mathbf{S} = \{\mathbf{x} \in (0,+\infty)^k: x_1 < x_2 < \dots < x_k \}$.
\end{definition}
This version of the Multivariate Gamma was introduced by \cite{mathal1992form}, motivated by the following construction:
Let $Z_i\sim \mathcal{G}(\alpha_i, \beta)$, iid random variables for $i \in [k]$, and define $X_i := \sum_{j=1}^i Z_j$. Then, $X_i\sim \mathcal{G}(\sum_{j=1}^i \alpha_j, \beta)$ and $\mathbf{X}\sim\mathcal{MG}_k(\boldsymbol{\alpha}, \beta)$. Lemma \ref{lemma:mgamma-properties} presents three well-known, essential properties. Proofs can be found, for example, in \cite{kotz2004continuous}. \par
\begin{lemma}\label{lemma:mgamma-properties}
    Let $\mathbf{X}\sim\mathcal{MG}_k(\boldsymbol{\alpha}, \beta)$ and $\Delta \mathbf{x} := (x_1, x_2-x_1, \dots, x_k-x_{k-1})$ be the first difference operator. Then,
    \begin{enumerate}[(i)]
        \item $\mathbf{Z} := \Delta \mathbf{X}$ is a random vector of iid variables $Z_i\sim\mathcal{G}(\alpha_i, \beta), \, i \in [k]$.
        \item $\mathbf{W} := \mathbf{Z} / X_k \sim \mathcal{D}_k(\boldsymbol{\alpha})$.
        \item $\mathbf{W}$ and $X_k$ are independent.
    \end{enumerate}
\end{lemma}
The random vectors $\mathbf{Z}$ and $\mathbf{W}$ introduced in Lemma \ref{lemma:mgamma-properties} simplify many expressions and therefore will be used throughout this section.
\begin{lemma}\label{mgamma-exp-fam}
The $\mathcal{MG}_k(\boldsymbol{\alpha}, \beta)$ distribution is a $k$-variate $(k+1)$-dimensional exponential family with 
\begin{equation*}
    \eta(\boldsymbol{\alpha}, \beta) = \left(\boldsymbol{\alpha}, -\frac{1}{\beta}\right), \quad
    T(\mathbf{X}) = \left(\log \mathbf{Z}, X_k\right), \quad
    A(\boldsymbol{\alpha}, \beta) = \sum_{i=1}^k\log \Gamma(\alpha_i) + \alpha_0\log\beta.
\end{equation*}
\end{lemma}
The necessary moments of $\mathbf{X}$ and $\mathbf{Z}$ are presented in the Lemma \ref{lemma:mgamma-moments}. It is underlined that, since $\mathbf{W}\sim\mathcal{D}_k(\boldsymbol{\alpha})$, its moments can be retrieved by the corresponding Dirichlet Lemmas \ref{lemma:dirichlet-moments}, \ref{appendix-lemma:dirichlet-moments-E} and \ref{appendix-lemma:dirichlet-moments-V}.
\begin{lemma}\label{lemma:mgamma-moments}
Let $\mathbf{X}\sim\mathcal{MG}_k(\boldsymbol{\alpha}, \beta)$. Then, for $i \in [k]$,
\vspace{-0.7cm}
\begin{equation*}
    \begin{minipage}{.30\linewidth}
      \centering
        \begin{align*}
            \EE(X_k) &= \alpha_0\beta, \\
            \EE(Z_i) &= \alpha_i\beta, \\
            \CC(Z_i, \log Z_i) &= \beta, 
        \end{align*}
    \end{minipage}
    \begin{minipage}{.60\linewidth}
      \centering 
        \begin{align*}
            \VV(X_k) &= \alpha_0\beta^2, \\
            \VV(Z_i) &= \alpha_i\beta^2, \\
            \EE(\log Z_i) &= \psi(\alpha_i) +\log\beta.
        \end{align*}
    \end{minipage}
\end{equation*}
\end{lemma}
\subsection{Estimators and Properties}
\label{ss:mgamma-estimators}
In this subsection, the formulas of the Multivariate Gamma estimators along with expressions of their asymptotic variance-covariance matrices are derived. As with Dirichlet, all proofs follow the same scheme. The components of $\mathbf{V}$ are moved to the Appendix Lemma \ref{appendix-lemma:mgamma-moments-V}.
\begin{proposition} \label{prop:mgamma-mle}
    Let $\mathbf{X}_n = (X_{1n}, \dots, X_{kn}), \, n \in [N]$ be a random sample from $\mathcal{MG}_k(\boldsymbol{\alpha}, \beta)$ and $\mathbf{Z}_n:=\Delta \mathbf{X}_n, \, n \in [N]$. The MLE of $(\boldsymbol{\alpha}, \beta)$, derived as the unique solution of the equation system
    \[
        \hat{\alpha}_0\hat{\beta} = \overline{X_k}, \qquad \overline{\log(Z_i)} = \psi(\hat{\alpha}_i) + \log \beta, \quad i \in [k],
    \]
    is a moment-type estimator with asymptotic variance-covariance matrix
    \[
    \left[ \begin{array}{cc}
       \mathbf{V}_{11} & \beta^{-1}\,\mathbf{1}_{k\times1} \\
       \beta^{-1}\,\mathbf{1}_{1\times k} & \alpha_0\beta^{-2} \\
    \end{array}\right]^{-1},
    \]
    where $\mathbf{V}_{11}$ is a diagonal matrix such that $\mathbf{V}_{11ii} = \psi_1(\alpha_i), \, i \in [k]$.
\end{proposition}
\begin{proof} 
    The result is immediate by Corollary \ref{corollary:mle}. Let $\lambda := \beta^{-1}$ and 
    \[
    h(\mathbf{x}):= (\log z_1, \dots, \log z_k, -x_k).
    \]
    In its canonical form, with parameters $\boldsymbol{\theta} = (\boldsymbol{\alpha}, \lambda)$, the asymptotic variance-covariance matrix takes the form $\mathbf{V}^{-1}$, where $\mathbf{V} = v_h(\boldsymbol{\alpha}, \beta)$. The matrix can be calculated from Lemma \ref{lemma:mgamma-moments} as
    \[
    \mathbf{V} = \left[ \begin{array}{cc}
       \mathbf{V}_{11} & -\lambda^{-1}\,\mathbf{1}_{k\times1} \\
       -\lambda^{-1}\,\mathbf{1}_{1\times k} & \alpha_0 \lambda^{-2} \\
    \end{array}\right],
    \]
    where $\mathbf{V}_{11}$ is a diagonal matrix such that $\mathbf{V}_{11ii} = \psi_1(\alpha_i), \, i \in [k]$. By applying the Delta method with $g(\mathbf{y}, y_{k+1}) := \left(\mathbf{y}, 1/ y_{k+1}\right)$, it is immediate to see that 
    \begin{align*}
        \frac{\partial g_i(\mathbf{y})}{\partial y_i} &= 1, \, i \in [k], \\
        \frac{\partial g_{k+1}(\mathbf{y})}{\partial y_{k+1}} &= -\frac{1}{y^2_{k+1}},
    \end{align*}
    therefore $\mathbf{G}=\nabla g(\boldsymbol{\theta})$ is a diagonal matrix with entries $\mathbf{G}_{ii} = 1, i \in [k]$ and $\mathbf{G}_{k+1\,k+1} = -\lambda^{-2}$. The final variance-covariance matrix takes the form $\mathbf{G^\top}\mathbf{V}^{-1}\mathbf{G}$, the inverse of which is easy to calculate due to the diagonality of $\mathbf{G}$:
    \[
       \left[ \begin{array}{cc}
       \mathbf{V}_{11} & \lambda\,\mathbf{1}_{k\times1} \\
       \lambda\,\mathbf{1}_{1\times k} & \alpha_0\lambda^{2} \\
        \end{array}\right].
    \]
    By substituting $\lambda$ for $\beta^{-1}$, the result follows.
\end{proof}
\begin{proposition} \label{prop:mgamma-me}
    Let $\mathbf{X}_n = (X_{1n}, \dots, X_{kn}), \, n \in [N]$ be a random sample from $\mathcal{MG}_k(\boldsymbol{\alpha}, \beta)$ and $\mathbf{Z}_n := \Delta \mathbf{X}_n, \, n \in [N]$. The Moment Estimator
    \[
        \tilde{\beta} = \frac{1}{k} \sum_{j=1}^k\frac{\overline{Z^2_j}-\overline{Z_j}^2}{\overline{Z_j}}, \quad
        \tilde{\alpha}_i = \frac{\overline{Z_i}}{\tilde{\beta}}, \quad i \in [k],
    \]
    is a moment-type estimator with functions
    \begin{align*}
        h(\mathbf{x}) &= (\mathbf{z}, \mathbf{z^2}), \\
        g_i(\mathbf{y}) &= \frac{y_i}{g_{k+1}(\mathbf{y})}, \, i \in [k], \\
        g_{k+1}(\mathbf{y}) &= \frac{1}{k} \sum_{j=1}^k\frac{y_{k+j}-y^2_j}{y_j},
    \end{align*}
    where $\mathbf{z}:=\Delta\mathbf{x}$. The asymptotic variance-covariance matrix takes the form $\mathbf{G}\mathbf{V}\mathbf{G}^\top$, where $\mathbf{V} = v_h(\boldsymbol{\alpha}, \beta)$ can be retrieved from Lemma \ref{appendix-lemma:mgamma-moments-V} and $\mathbf{G}$ is a $(k+1)\times 2k$ block matrix:
    \[
        \mathbf{G} = 
        \left[\begin{matrix}
            \mathbf{G}_{11} & \mathbf{G}_{12} \\ 
            \mathbf{G}_{21} & \mathbf{G}_{22}
        \end{matrix}\right],
    \]
    such that,  for $i, j \in [k]$:
    \vspace{-0.5cm}
    \begin{equation*}
        \begin{minipage}{.50\linewidth}
          \centering
            \begin{align*}
            \mathbf{G}_{11ij} &= \frac{\alpha_i}{k\beta}\left(2 + \frac{1}{\alpha_j}\right) + \frac{1}{\beta}\mathbb{I}_{\{i=j\}} \\
            \mathbf{G}_{21j} &= -\frac{1}{k}\left(2 + \frac{1}{\alpha_j}\right),
            \end{align*}
        \end{minipage}%
        \begin{minipage}{.50\linewidth}
          \centering
            \begin{align*}
            \mathbf{G}_{12ij} &= -\frac{\alpha_i}{k\alpha_j\beta^2} \\
            \mathbf{G}_{22j} &= \frac{1}{k\alpha_j\beta}.
            \end{align*}
        \end{minipage}%
    \end{equation*}
\end{proposition}
\begin{proof} 
    It is immediate to see that Definition \ref{def:me} is satisfied by functions $h$ and $g$, therefore $(\tilde{\boldsymbol{\alpha}}, \tilde{\beta})$ is indeed a moment-type estimator. The $(k+1)\times 2k$ Jacobian matrix of $g$ can be broken down into blocks, such that for $i, j \in [k]$:
    \vspace{-0.5cm}
    \begin{equation*}
        \begin{minipage}{.50\linewidth}
          \centering
            \begin{align*}
            \frac{\partial g_i}{\partial y_j}(\mathbf{y}) &=  \frac{ky_i}{c^2(\mathbf{y})}\left(\frac{y_{k+j}}{y^2_j}+1\right) + \frac{k}{c(\mathbf{y})} \mathbb{I}_{\{i=j\}}, \\
            \frac{\partial g_{k+1}}{\partial y_j}(\mathbf{y}) &= -k\left(\frac{y_{k+j}}{y^2_j}+1\right),
            \end{align*}
        \end{minipage}%
        \begin{minipage}{.40\linewidth}
          \centering 
            \begin{align*}
                \frac{\partial g_i}{\partial y_{k+j}}(\mathbf{y}) &= -\frac{ky_i}{y_jc^2(\mathbf{y})}, \\
                \frac{\partial g_{k+1}}{\partial y_{k+j}}(\mathbf{y}) &= \frac{1}{ky_j},
            \end{align*}
        \end{minipage}
    \end{equation*}
    where $c(\mathbf{y}) := \sum_{j=1}^k(y_{k+j}/y_j-y_j)$.
    By evaluating the derivatives in $\boldsymbol{\mu} = e_h(\boldsymbol{\alpha}, \beta)$ and retrieving the required moments from Lemma \ref{appendix-lemma:mgamma-moments-E} the form of $\mathbf{G}$ follows. 
\end{proof}
\begin{theorem} \label{th:mgamma-same}
    Let $\mathbf{X}_n = (X_{1n}, \dots, X_{kn}), \, n \in [N]$, be a random sample from $\mathcal{MG}_k(\boldsymbol{\alpha}, \beta)$ and $\mathbf{Z}_n := \Delta \mathbf{X}_n, \, n \in [N]$. The Score-Adjusted Moment Estimator
    \[
        \breve{\beta} = \frac{1}{k} \sum_{j=1}^k\left[\,\overline{Z_j \log Z_j}-\overline{Z_j} \, \overline{\log Z_j}\,\right], \quad
        \breve{\alpha_i} = \frac{\overline{Z_i}}{\breve{\beta}}, \quad i \in [k],
    \]
    is a moment-type estimator with functions
    \begin{align*}
        h(\mathbf{x}) &= (z_1, \dots, z_k, \log z_1, \dots, \log z_k, z_1 \log z_1, \dots, z_k \log z_k), \\
        g_i(\mathbf{y}) &= \frac{y_i}{g_{k+1}(\mathbf{y})}, \, i \in [k], \\
        g_{k+1}(\mathbf{y}) &= \frac{1}{k} \sum_{j=1}^k(y_{2k+j} - y_{j}y_{k+j}),
    \end{align*}
    where $\mathbf{z}:=\Delta\mathbf{x}$. The asymptotic variance-covariance matrix takes the form $\mathbf{G}\mathbf{V}\mathbf{G}^\top$, where $\mathbf{V} = v_h(\boldsymbol{\alpha}, \beta)$ can be retrieved from Lemma \ref{appendix-lemma:mgamma-moments-V} and $\mathbf{G}$ is a $(k+1)\times 3k$ block matrix: \\
    \[
        \mathbf{G} = 
        \left[\begin{matrix}
            \mathbf{G}_{11} & \mathbf{G}_{12} & \mathbf{G}_{13} \\ 
            \mathbf{G}_{21} & \mathbf{G}_{22} & \mathbf{G}_{23}
        \end{matrix}\right],
    \]
    such that,  for $i, j \in [k]$:
    \vspace{-0.5cm}
    \begin{equation*}
        \begin{minipage}{.30\linewidth}
          \centering
            \begin{align*}
            \mathbf{G}_{11ij}&= \frac{\alpha_i}{k\beta}\left[\psi(\alpha_j)+\log\beta\right] + \frac{1}{\beta}\mathbb{I}_{\{i=j\}}, \\
            \mathbf{G}_{21j\,} &= -\frac{1}{k}\left[\psi(\alpha_j)+\log\beta\right],
            \end{align*}
        \end{minipage}%
        \begin{minipage}{.30\linewidth}
          \centering
            \begin{align*}
            \mathbf{G}_{12ij} &= \frac{\alpha_i\alpha_j}{k}, \\
            \mathbf{G}_{22j\,} &= -\frac{1}{k}\alpha_j\beta
            \end{align*}
        \end{minipage}%
        \begin{minipage}{.30\linewidth}
          \centering 
            \begin{align*}
                \mathbf{G}_{13ij} &= -\frac{\alpha_i}{k\beta}, \\
                \mathbf{G}_{23j\,} &= \frac{1}{k}.
            \end{align*}
        \end{minipage}
    \end{equation*}
\end{theorem}
\begin{proof} 
    It is immediate to see that Definition \ref{def:me} is satisfied by functions $h$ and $g$, therefore $(\breve{\boldsymbol{\alpha}}, \breve{\beta})$ is indeed a moment-type estimator. The $(k+1)\times 3k$ Jacobian matrix of $g$ can be broken down into three $k\times k$ blocks and three $k\times1$ column-matrices such that, for $i, j \in [k]$:
    \vspace{-0.5cm}
    \begin{equation*}
        \begin{minipage}{.30\linewidth}
          \centering
            \begin{align*}
            \frac{\partial g_i}{\partial y_j}(\mathbf{y}) &= \frac{ky_iy_{k+j}}{c^2(\mathbf{y})} + \frac{k}{c(\mathbf{y})} \mathbb{I}_{\{i=j\}}, \\
            \frac{\partial g_{k+1}}{\partial y_j}(\mathbf{y}) &= -\frac{1}{k}y_{k+j},
            \end{align*}
        \end{minipage}%
        \begin{minipage}{.30\linewidth}
          \centering
            \begin{align*}
            \frac{\partial g_i}{\partial y_{k+j}}(\mathbf{y}) &= \frac{ky_iy_j}{c^2(\mathbf{y})}, \\
            \frac{\partial g_{k+1}}{\partial y_{k+j}}(\mathbf{y}) &= -\frac{1}{k}y_j,
            \end{align*}
        \end{minipage}%
        \begin{minipage}{.30\linewidth}
          \centering 
            \begin{align*}
                \frac{\partial g_i}{\partial y_{2k+j}}(\mathbf{y}) &= -\frac{ky_i}{c^2(\mathbf{y})}, \\
                \frac{\partial g_{k+1}}{\partial y_{2k+j}}(\mathbf{y}) &= \frac{1}{k},
            \end{align*}
        \end{minipage}
    \end{equation*}
    where $c(\mathbf{y}) := \sum_{j=1}^k(y_{2k+j}-y_jy_{k+j})$. By evaluating the derivatives in $\boldsymbol{\mu} = e_h(\boldsymbol{\alpha}, \beta)$ and retrieving the required moments from Lemma \ref{appendix-lemma:mgamma-moments-E} the form of $\mathbf{G}$ follows.
\end{proof}
\begin{theorem}
    Let $\mathbf{X}_n = (X_{1n}, \dots, X_{kn}), \, n \in [N]$, be a random sample from $\mathcal{MG}_k(\boldsymbol{\alpha}, \beta)$ and $\mathbf{Z}_n := \Delta \mathbf{X}_n, \, n \in [N]$. Then, $n\breve{\beta}/(n-1)$ is an unbiased estimator of $\beta$ and $n\breve{\alpha}^{-1}_i/(n-1)$ is an unbiased estimator of $\alpha^{-1}_i$ for $i \in [k]$.
\end{theorem}
\begin{proof}
    The result follows from the corresponding properties of the univariate Gamma SAME, a proof for which is given by \cite{ye2017closed}. For $\breve{\beta}$, 
    \[
        \EE\left(\breve{\beta}\right) = \frac{1}{k} \sum_{j=1}^k \EE\left(\overline{Z_j \log Z_j}-\overline{Z_j} \, \overline{\log Z_j}\right),
    \]
    and the within-sum expectation is identical to the univariate Gamma case, shown to be equal to $(n-1)\beta/n$ (it can also be calculated from Lemma \ref{appendix-lemma:mgamma-moments-E}). Therefore,
    \[
    \EE\left(\breve{\beta}\right) = \frac{1}{k}\sum_{i=1}^k \frac{(n-1)\beta}{n} = \frac{(n-1)\beta}{n},
    \]
    and $n\breve{\beta}/(n-1)$ is an unbiased estimator of $\beta$. Similarly, $\breve{\alpha_i}$ is independent from $\overline{Z_i}$ \citep{ye2017closed}, therefore
    \[
    \EE\left(\frac{\overline{Z_i}}{\breve{\alpha_i}}\right) = \EE\left(\overline{Z_i}\right)\EE\left(\breve{\alpha}^{-1}_i\right) = \alpha_i\beta\EE\left(\breve{\alpha}^{-1}_i\right), \quad i \in [k],
    \]
    but $\overline{Z_i}/\breve{\alpha_i} = \breve{\beta}$, thus
    \[
    \alpha_i\beta\EE\left(\breve{\alpha}^{-1}_i\right) = \frac{(n-1)\beta}{n} \Longrightarrow \EE\left(\breve{\alpha_i}\right) = \frac{n-1}{n}\alpha^{-1}_i, \quad i \in [k],
    \]
    and $n\breve{\alpha}^{-1}_i/(n-1)$ is an unbiased estimator of $\alpha^{-1}_i$ for $i \in [k]$.
\end{proof}
We now introduce a special class of estimators for the Multivariate Gamma distribution, best coined as \textit{Dirichlet-based} moment-type estimators.
\begin{theorem} \label{th:dirichlet-based-me}
    Let $\mathbf{X}_n = (X_{1n}, \dots, X_{kn}), \, n \in [N]$ be a random sample from $\mathcal{MG}_k(\boldsymbol{\alpha}, \beta)$ and $\tilde{\boldsymbol{\alpha}}$ be the moment-type estimator derived by the corresponding $\mathcal{D}_k(\boldsymbol{\alpha})$ sample $\mathbf{W}_n, \, n \in [N]$. Then, the estimator $\tilde{\boldsymbol{\theta}} = (\tilde{\boldsymbol{\alpha}}, \tilde{\beta})$, where $\tilde{\beta} = \overline{X_k} / \sum_{i=1}^k \tilde{\alpha}_i$, is a moment-type estimator with asymptotic variance-covariance matrix $\mathbf{G}\mathbf{V}\mathbf{G}^\top$, such that:
    \[
    \mathbf{G} = \left[ \begin{array}{cc}
       \mathbf{G}_D & \mathbf{O}_{k\times 1} \\
       \beta\alpha^{-1}_0 \mathbf{1}_{1\times m} \mathbf{G}_D & \alpha^{-1}_0 \\
    \end{array}\right], \quad
    \mathbf{V} = \left[ \begin{array}{cc}
       \mathbf{V}_D & \mathbf{O}_{m\times 1} \\
       \mathbf{O}_{1\times m} & \alpha_0\beta^2 \\
    \end{array}\right],
    \]
    where $\mathbf{G}_D, \mathbf{V}_D$ are the corresponding matrices for $\tilde{\boldsymbol{\alpha}}$.
\end{theorem}
\begin{proof}
    Denote $h_D$ and $g_D$ the functions used to derive the moment-type estimator $\tilde{\boldsymbol{\alpha}}$ based on $\mathbf{W}$. Let $\mathbf{X}^\star:= (\mathbf{W}, X_k)$ and $h, g$ such that:
    \begin{align*}
        h(\mathbf{x}) &:= \left(h_{D}(\mathbf{x}), x_k\right), \\ 
        g(\mathbf{y}, y_{m+1}) &:= \left(g_{D}(\mathbf{y}), \frac{y_{m+1}}{\sum_{i=1}^kg_{D_i}(\mathbf{y})}\right).
    \end{align*}
    It is straightforward to see that this construction leaves the estimator of $\boldsymbol{\alpha}$ unchanged, $\tilde{\boldsymbol{\alpha}} = g_D\left(\overline{h_D\left(\mathbf{W}_n\right)}\right)$, and the last coordinate results in
    \[
         \frac{\overline{X_k}}{\sum_{i=1}^k g_{D_i}\left(\overline{h_D\left(\mathbf{W}_n\right)}\right)} = \frac{\overline{X_k}}{\tilde{\alpha}_0} = \tilde{\beta},
    \]
    therefore $\tilde{\boldsymbol{\theta}}_n = g\left(\overline{h\left(\mathbf{X}^\star_n\right)}\right)$ is indeed a moment-type estimator.
    The asymptotic variance-covariance matrix $\mathbf{G}\mathbf{V}\mathbf{G}^\top$ is such that:
    \[
    \mathbf{G} = \left[ \begin{array}{cc}
       \mathbf{G}_D & \mathbf{G}_{12} \\
       \mathbf{G}_{21} & \mathbf{G}_{22} \\
    \end{array}\right], \quad
    \mathbf{V} = \left[ \begin{array}{cc}
       \mathbf{V}_D & \mathbf{V}_{12} \\
       \mathbf{V}_{21} & \mathbf{V}_{22} \\
    \end{array}\right],
    \]
    where $\mathbf{G}_D, \mathbf{V}_D$ are the corresponding matrices for $\tilde{\boldsymbol{\alpha}}$. The derivatives of $g$ take the form:
    \begin{align}
        \frac{\partial g_{D_i}(\mathbf{y})}{\partial y_{m+1}} &= 0, \quad i \in [k], \label{eq:mgamma-der1} \\
        \frac{\partial g_{D_{k+1}}(\mathbf{y})}{\partial y_{j}} &= -\frac{y_{m+1}}{\left[\sum_{i=1}^k g_{D_i}(\mathbf{y})\right]^2} \sum_{i=1}^k\frac{\partial g_{D_i}(\mathbf{y})}{\partial y_{j}}, \quad j \in [m], \label{eq:mgamma-der2} \\
        \frac{\partial g_{D_{k+1}}(\mathbf{y})}{\partial y_{m+1}} &= \frac{1}{\sum_{i=1}^k g_{D_i}(\mathbf{y})}. \label{eq:mgamma-der3}
    \end{align}
    Formula (\ref{eq:mgamma-der1}) results in $\mathbf{G}_{12} = \mathbf{O}_{1\times k}$. Formula (\ref{eq:mgamma-der2}), evaluated in $\boldsymbol{\mu}=\EE(\mathbf{X}^\star)$, can be simplified since 
    \[
        \frac{\mu_{m+1}}{\left[\sum_{i=1}^k g_{D_i}(\boldsymbol{\mu})\right]^2} = \frac{\EE(X_k)}{\left(\sum_{i=1}^k \alpha_i\right)^2} = \frac{\alpha_0\beta}{\alpha_0^2} = \frac{\beta}{\alpha_0}.
    \]
    The expression can further be written in matrix form by noticing that the sum involved is the column-sum of the Jacobian matrix $\mathbf{G}_D = \nabla g_D(\boldsymbol{\mu})$, which can be expressed as a left-multiplication with the row matrix $\mathbf{1}_{1\times m}$, resulting in $\mathbf{G}_{21} = \beta\alpha^{-1}_0 \mathbf{1}_{1\times m} \mathbf{G}_D $. Formula (\ref{eq:mgamma-der3}), evaluated in $\boldsymbol{\mu}$, gives $\mathbf{G}_{22} = \alpha^{-1}_0$. The form of $\mathbf{G}$ follows. \par
    From Lemma \ref{lemma:mgamma-properties}, $\mathbf{Y}$ and $X_k$ are independent, therefore $\CC(h_{D_i}(\mathbf{Y}), X_k) = 0,\ i \in [m]$, and $\mathbf{V}_{12}$, $\mathbf{V}_{21}$ are zero matrices. Furthermore, from Lemma \ref{lemma:mgamma-moments}, $\VV(X_k) = \alpha_0\beta^2$. The form of $\mathbf{V}$ follows.
\end{proof}
Building upon Theorem \ref{th:dirichlet-based-me}, it is immediate to derive the asymptotic variance-covariance matrix form of any Dirichlet-based moment-type estimator. Corollary \ref{cor:mgamma-dme} introduces the Dirichlet-based ME and SAME estimators for the Multivariate Gamma distribution.
\begin{corollary}\label{cor:mgamma-dme}
    Let $\mathbf{X}_n = (X_{1n}, \dots, X_{kn}),\ n \in [N]$ be a random sample from $\mathcal{MG}_k(\boldsymbol{\alpha}, \beta)$ and $\mathbf{W}_n$ be the corresponding $\mathcal{D}_k(\boldsymbol{\alpha})$ sample. Then, the Dirichlet-based ME $(\tilde{\boldsymbol{\alpha}}, \tilde{\beta})$ and the Dirichlet-based SAME $(\breve{\boldsymbol{\alpha}}, \breve{\beta})$ are moment-type estimators with asymptotic variance-covariance matrices of the form derived in Theorem \ref{th:dirichlet-based-me}.
\end{corollary}
\begin{remark}
    The Multivariate Gamma MLE equations of Proposition \ref{prop:mgamma-mle} can be reshaped into $\overline{\log(W_i)} = \psi(\hat{\alpha}_i) - \log \hat{\alpha}_0, \, i \in [k]$, which is \textit{not} the same as the corresponding Dirichlet-based MLE equations of Proposition \ref{prop:dirichlet-mle}, $\overline{\log(W_i)} = \psi(\hat{\alpha}_i) - \psi(\hat{\alpha}_0), \, i \in [k]$. However, it is worth noting that due to the inequality $\log(\alpha) - 1/\alpha\leq\psi(\alpha)\leq\log(\alpha) - 1/2\alpha, \, \alpha\in(0, +\infty)$, the solutions of the two systems are expected to be close for large values of $\alpha_0$. Since the latter is neither asymptotically efficient nor explicit, it is not included in the simulations.
\end{remark}
\subsection{Simulation Study}
\label{ss:mgamma-study}
\begin{figure}[ht]
    \centering
    \includegraphics[width=\textwidth]{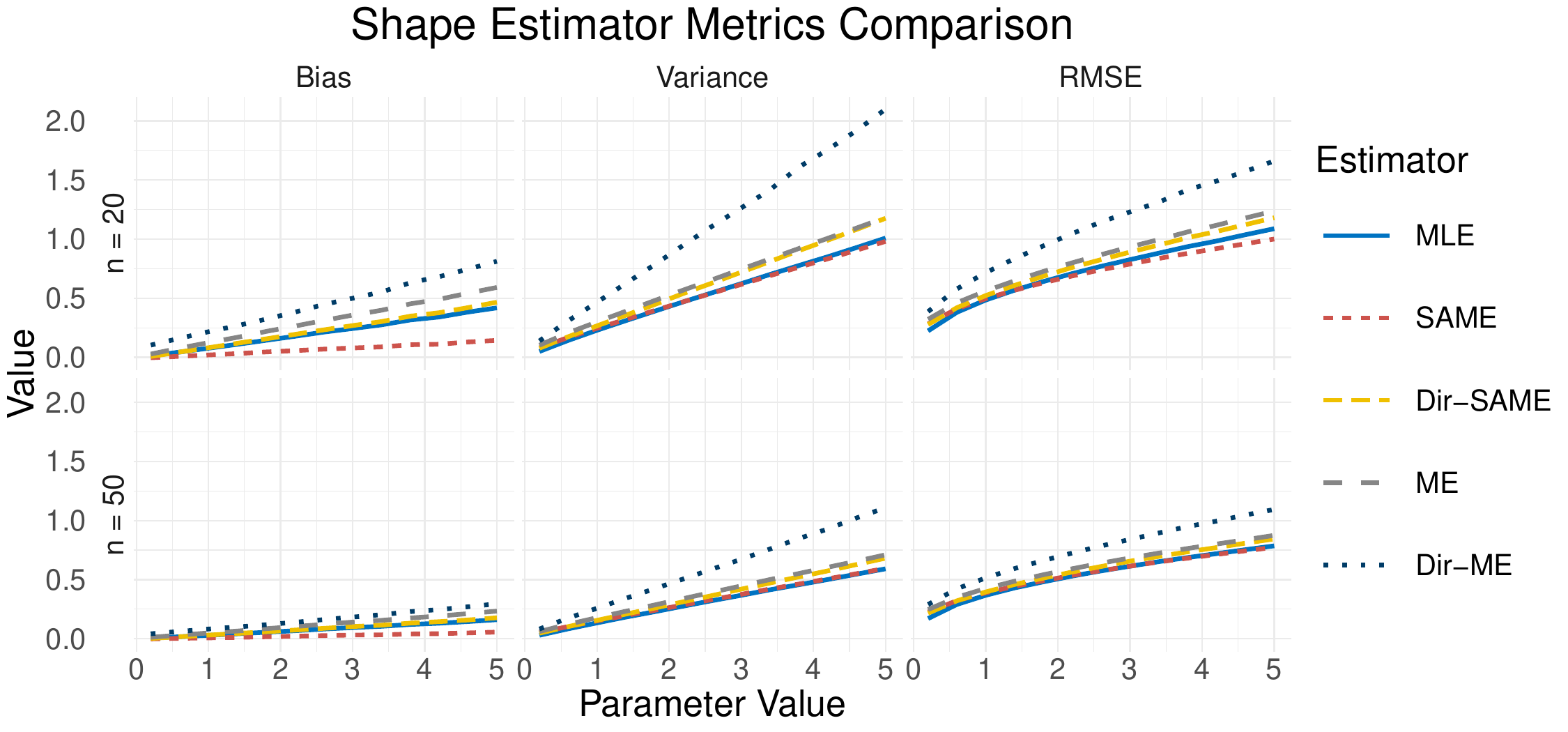}
    \caption{Metrics of the $\alpha_1$ parameter estimators in a 5-dimensional Multivariate Gamma distribution with $\boldsymbol{\alpha} = (\alpha_1, 1, 2, 5), \, \beta = 1$ and $\alpha_1 \in [0.2, 5]$. The metrics were estimated for $n = 20$ and $n = 50$ observations with a Monte Carlo sample size $m = 10^5$.}
    \label{fig:mgamma-metrics-shape1}
\end{figure}
\begin{figure}[ht!]
    \centering
    \includegraphics[width=\textwidth]{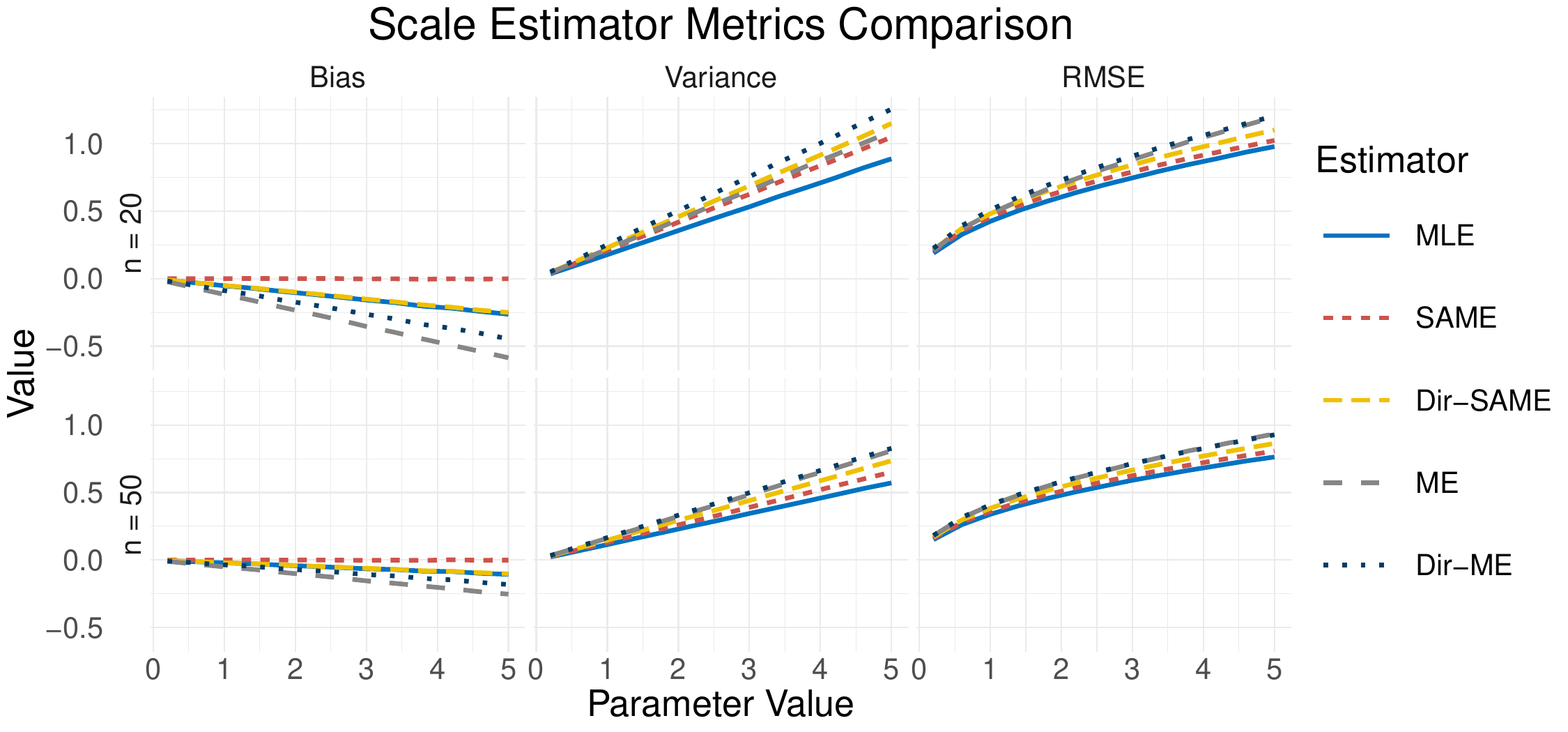}
    \caption{Metrics of the $\beta$ parameter estimators in a 5-dimensional Multivariate Gamma distribution with $\boldsymbol{\alpha} = (0.2, 1, 2, 5)$ and $\beta \in [0.2, 5]$. The metrics were estimated for $n = 20$ (upper row) and $n = 50$ (lower row) observations with a Monte Carlo sample size $m = 10^5$.}
    \label{fig:mgamma-metrics-scale}
\end{figure}
This subsection presents the Multivariate Gamma estimator comparison, following the same scheme as the Dirichlet simulation study. Specifically, the shape and scale estimator metrics were estimated with $m = 10^5$ samples of size $n = 20$ and $n = 50$. For the shape, samples were generated from a 5-dimensional Multivariate Gamma distribution with $\boldsymbol{\alpha} = (\alpha_1, 1, 2, 5)$, $\alpha_1 \in [0.2, 5]$ and $\beta = 1$. For the scale, samples were generated from a 5-dimensional Multivariate Gamma distribution with $\boldsymbol{\alpha} = (0.2, 1, 2, 5)$ and $\beta \in [0.2, 5]$. The results are presented in Figures \ref{fig:mgamma-metrics-shape1} and \ref{fig:mgamma-metrics-scale}, respectively. It is evident that the SAME and Dirichlet-based SAME outperform the classic ME. Specifically, the SAME estimator illustrates the smallest bias in absolute value and performs particularly well for small sample sizes, with the shape estimator surpassing the MLE in terms of RMSE (Figure \ref{fig:mgamma-metrics-shape1}, upper-right graph). \par
\begin{figure}[ht]
    \centering
    \includegraphics[width=\textwidth]{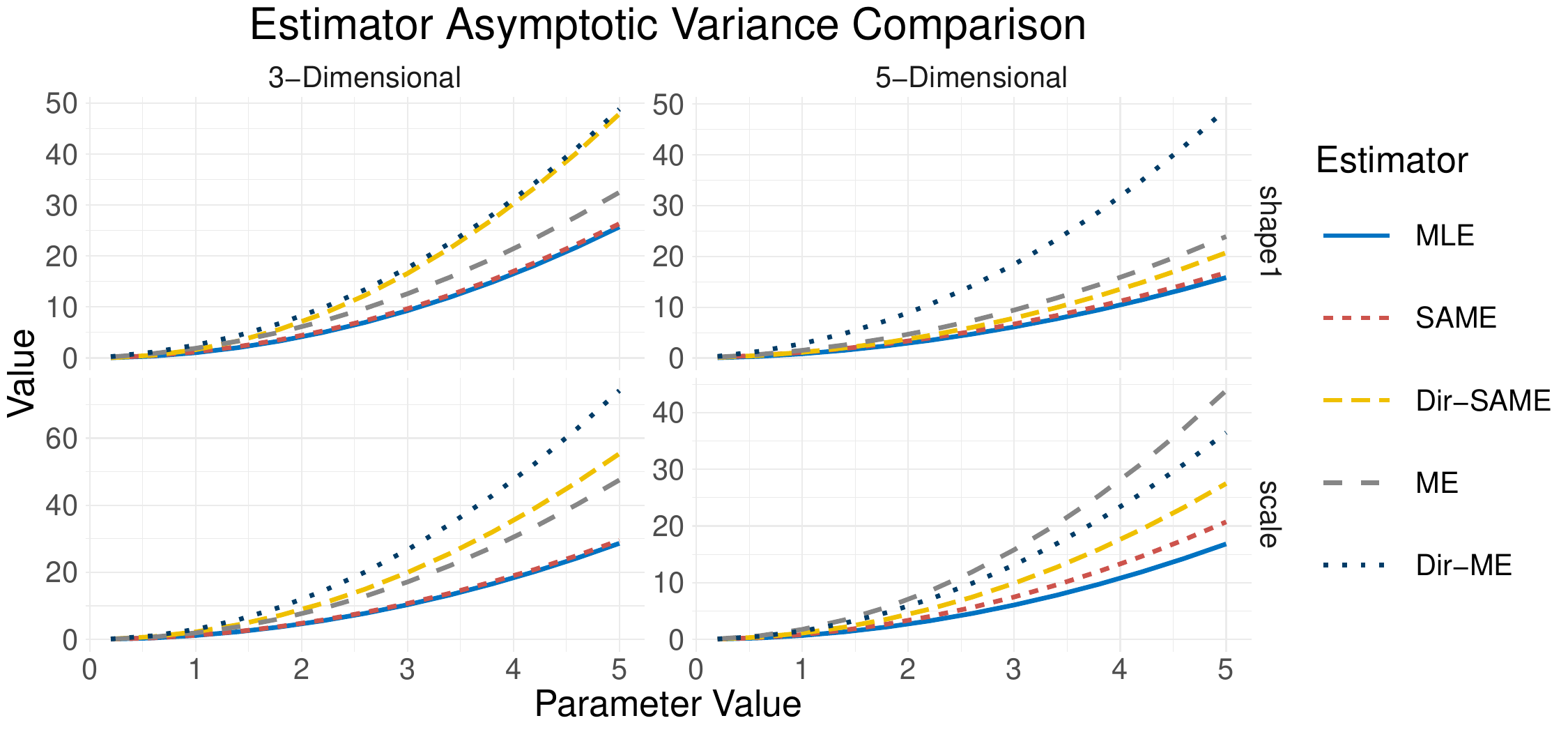}
    \caption{Upper row: Asymptotic variance of the $\alpha_1\in [0.2, 5]$ parameter estimators in a 3-dimensional Multivariate Gamma with $\boldsymbol{\alpha} = (\alpha_1, 5), \,\beta = 1$ and a 5-dimensional Multivariate Gamma with $\boldsymbol{\alpha} = (\alpha_1, 1, 2, 5), \, \beta = 1$. Lower row: Asymptotic variance of the $\beta\in [0.2, 5]$ parameter estimators in a 3-dimensional Multivariate Gamma with $\boldsymbol{\alpha} = (1, 5),$ and a 5-dimensional Multivariate Gamma with $\boldsymbol{\alpha} = (0.2, 1, 2, 5)$.}
    \label{fig:mgamma-acov}
\end{figure}
The asymptotic efficiency of the shape estimators was examined for $\alpha_1\in [0.2, 5]$ in a 3-dimensional Multivariate Gamma with $\boldsymbol{\alpha} = (\alpha_1, 5), \,\beta = 1$ and a 5-dimensional Multivariate Gamma with $\boldsymbol{\alpha} = (\alpha_1, 1, 2, 5), \, \beta = 1$. Similarly, the asymptotic variance of the scale estimators was examined for $\beta\in [0.2, 5]$ in a 3-dimensional Multivariate Gamma with $\boldsymbol{\alpha} = (1, 5),$ and a 5-dimensional Multivariate Gamma with $\boldsymbol{\alpha} = (0.2, 1, 2, 5)$. The results are presented in Figure \ref{fig:mgamma-acov}. In both cases, the SAME and Dir-SAME perform better than the classic moment estimators. 
\section{Conclusions}
\label{s:conclusions}
This research paper introduced novel closed-form estimators for the Dirichlet and Multivariate Gamma distribution families, filling the gap where explicit maximum likelihood estimators are not attainable. These estimators build upon successful concepts designed for the Beta and Gamma distributions, directly extending them to their multivariate generalizations. The purpose of this study was to present the new estimators as members of the moment-type class, highlighting their common foundation and obviating the need for repetitive proofs. Furthermore, since the efficient coding of such estimators is cumbersome and error-prone, an R package was developed to allow the direct application of the methodology developed. \par 
The search for explicit estimators in families for which the MLE cannot be derived in closed form holds a great scientific interest especially when the parameter dimension is high. An extension of the current study could focus on the matrix-variate analogs of the Beta and Gamma families, which also includes the commonly used Wishart distribution family. Another direction that requires separate analysis is the bias correction for the estimators already developed.
\begin{center}
{\large\bf SUPPLEMENTARY MATERIAL}
\end{center}
\begin{description}
    \item[R package:] Package \texttt{estimators} is publicly available on CRAN. The development version of the package is available on GitHub in the github.com/thechibo/estimators repository. Information regarding the package's use can be found in the relevant documentation. 
\end{description}
\bibliographystyle{agsm}
\bibliography{bibliography}
\newpage
\appendix
\section{Appendix}
\label{appendix:tools}
\begin{lemma} \label{appendix-lemma:polygamma}
    For the digamma difference function, the following formulas hold:
    \begin{align*}
    \Psi(x+1, x) &= \frac{1}{x}, \\
    \Psi(x+1, y) &= \Psi(x, y) + \frac{1}{x}, \\
    \Psi(x, y+1) &= \Psi(x, y) - \frac{1}{y}.
    \end{align*}
\end{lemma}
\begin{lemma}
    For the gamma function, the following derivate formulas hold:
    \begin{align*}
        \frac{\partial}{\partial\alpha_i}\left(\frac{\Gamma(\alpha_i)}{\Gamma(\alpha_0)}\right) &= \frac{\Gamma(\alpha_i)}{\Gamma(\alpha_0)} \Psi(\alpha_i, \alpha_0), \\
        \frac{\partial^2}{\partial\alpha^2_i}\left(\frac{\Gamma(\alpha_i)}{\Gamma(\alpha_0)}\right) &= \frac{\Gamma(\alpha_i)}{\Gamma(\alpha_0)} \left[\Psi^2(\alpha_i, \alpha_0)+\Psi_1(\alpha_i, \alpha_0)\right], \\
        \frac{\partial^2}{\partial\alpha_i\partial\alpha_j}\left(\frac{\Gamma(\alpha_i)\Gamma(\alpha_j)}{\Gamma(\alpha_0)}\right) &= \frac{\Gamma(\alpha_i)\Gamma(\alpha_j)}{\Gamma(\alpha_0)} \left[\Psi(\alpha_i, \alpha_0)\Psi(\alpha_j, \alpha_0)-\psi_1(\alpha_0)\right], \\
        \frac{\partial}{\partial\alpha}\left(\Gamma(\alpha+m)\beta^{\alpha+m}\right) &= \Gamma(\alpha+m) \beta^{\alpha+m}\left[\psi(\alpha+m)+\log\beta\right], \\
        \frac{\partial^2}{\partial\alpha^2}\left(\Gamma(\alpha+m)\beta^{\alpha+m}\right) &= \Gamma(\alpha+m) \beta^{\alpha+m}\left\{\psi_1(\alpha+m) + \left[\psi(\alpha+m)+\log\beta\right]^2\right\}.
    \end{align*}
\end{lemma}
\begin{lemma} \label{appendix-lemma:dirichlet-moments-E}
    Let $\mathbf{X}=(X_1, \dots, X_k)\sim\mathcal{D}_k(\boldsymbol{a})$. Then, assuming $i, j \in [k], i\neq j$:
    {\small
    \begin{align*}
        \EE\left(X^m_i\right) &= \frac{\alpha_i\dots(\alpha_i+m-1)}{\alpha_0\dots(\alpha_0+m-1)}, \, m\in\mathbb{N}, \\
        \EE\left(X^{m_i}_iX^{m_j}_j\right) &= \frac{\alpha_i\dots(\alpha_i+m_i-1)\alpha_j\dots(\alpha_j+m_j-1)}{\alpha_0\dots(\alpha_0+m_i+m_j-1)}, \, m_i, m_j\in\mathbb{N}, \\
        \EE(\log X_i) &= \Psi(\alpha_i, \alpha_0), \\
        \EE\left(X_i\log X_i\right) &= \frac{\alpha_i}{\alpha_0}\Psi(\alpha_i+1, \alpha_0+1), \\
        \EE\left(X_i\log X_j\right) &= \frac{\alpha_i}{\alpha_0}\Psi(\alpha_j, \alpha_0+1), \\
        \EE\left(X_iX_j\log X_j\right) &= \frac{\alpha_i\alpha_j}{\alpha_0(\alpha_0+1)}\Psi(\alpha_j+1, \alpha_0+2), \\
        \EE\left(\log X_i\log X_j\right) &= \Psi(\alpha_i, \alpha_0)\Psi(\alpha_j, \alpha_0)-\psi_1(\alpha_0), \\
        \EE\left(X_i\log X_i\log X_j\right) &= \frac{\alpha_i}{\alpha_0}\left[\Psi(\alpha_i+1, \alpha_0+1)\Psi(\alpha_j, \alpha_0+1)-\psi_1(\alpha_0+1)\right], \\
        \EE\left(X_iX_j\log X_i\log X_j\right) &= \frac{\alpha_i\alpha_j}{\alpha_0(\alpha_0+1)}\left[\Psi(\alpha_i+1, \alpha_0+2)\Psi(\alpha_j+1, \alpha_0+2)-\psi_1(\alpha_0+2)\right], \\
        \EE\left(X^2_i\log X_i\right) &= \frac{\alpha_i(\alpha_i+1)}{\alpha_0(\alpha_0+1)}\Psi(\alpha_i+2, \alpha_0+2), \\
        \EE\left(X_i\log^2 X_i\right) &= \frac{\alpha_i}{\alpha_0}\left[\Psi^2(\alpha_i+1, \alpha_0+1)+\Psi_1(\alpha_i+1, \alpha_0+1)\right], \\
        \EE\left(X^2_i\log^2 X_i\right) &= \frac{\alpha_i(\alpha_i+1)}{\alpha_0(\alpha_0+1)}\left[\Psi^2(\alpha_i+2, \alpha_0+2)+\Psi_1(\alpha_i+2, \alpha_0+2)\right].
    \end{align*}
    }
\end{lemma}
\begin{lemma} \label{appendix-lemma:dirichlet-moments-V}
    Let $\mathbf{X}=(X_1, \dots, X_k)\sim\mathcal{D}_k(\boldsymbol{a})$. Then, assuming $i, j \in [k], i\neq j$:
    {\small
    \begin{align*}
        \VV(X^m_i) &= \frac{\alpha_i\dots(\alpha_i+2m-1)}{\alpha_0\dots(\alpha_0+2m-1)} - \left[\frac{\alpha_i\dots(\alpha_i+m-1)}{\alpha_0\dots(\alpha_0+m-1)}\right]^2, \, m\in\mathbb{N}, \\
        \VV(\log X_i) &= \Psi_1(\alpha_i, \alpha_0), \\
        \CC(X_i, X_j) &= -\frac{\alpha_i\alpha_j}{\alpha_0^2(\alpha_0+1)}, \\
        \CC(X_i, X^2_i) &= \frac{2\alpha_i\beta_i(\alpha_i+1)}{\alpha_0^2(\alpha_0+1)(\alpha_0+2)}, \\
        \CC(X_i, X^2_j) &= -\frac{2\alpha_i\alpha_j(\alpha_j+1)}{\alpha_0^2(\alpha_0+1)(\alpha_0+2)}, \\
        \CC(X^2_i, X^2_j) &= -\frac{2\alpha_i(\alpha_i+1)\alpha_j(\alpha_j+1)(2\alpha_0+3)}{\alpha_0^2(\alpha_0+1)^2(\alpha_0+2)(\alpha_0+3)}, \\
        \CC(X_i, \log X_i) &= \frac{\beta_i}{\alpha_0^2}, \\
        \CC(X_i, \log X_j) &= -\frac{\alpha_i}{\alpha_0^2}, \\
        \CC(\log X_i, \log X_j) &= -\psi_1(\alpha_0),  \\
        \CC(X_i, X_i\log X_i) &= \frac{\alpha_i\beta_i}{\alpha^2_0(\alpha_0+1)}\left[\Psi(\alpha_i+1,\alpha_0+2)+1\right], \\
        \CC(X_i, X_j\log X_j) &= -\frac{\alpha_i\alpha_j}{\alpha^2_0(\alpha_0+1)}\left[\Psi(\alpha_j+1,\alpha_0+2)+1\right], \\
        \CC(\log X_i, X_i\log X_i) &= \frac{\beta_i}{\alpha^2_0}\Psi(\alpha_i+1,\alpha_0+1) + \frac{\alpha_i}{\alpha_0}\Psi_1(\alpha_i+1,\alpha_0+1), \\
        \CC(\log X_i, X_j\log X_j) &= -\frac{\alpha_j}{\alpha^2_0}\Psi(\alpha_j+1,\alpha_0+1) - \frac{\alpha_j}{\alpha_0}\psi_1(\alpha_0+1).
    \end{align*}
    }
\end{lemma}
\begin{lemma} \label{appendix-lemma:mgamma-moments-E}
    Let $\mathbf{X}=(X_1, \dots, X_k)\sim\mathcal{MG}_k(\boldsymbol{a}, \beta)$ and $\mathbf{Z}= \Delta\mathbf{X}$. Then, for $i \in [k], m, \in \mathbb{N}$:
    {\small
    \begin{align*}
    \EE(Z_i^m) &= \beta^m \left[\alpha_i\dots(\alpha_i+m-1)\right], \\
    \EE(Z_i^m\log Z_i) &= \beta^m \left[\alpha_i\dots(\alpha_i+m-1)\right] \left[\psi(\alpha_i+m)+\log\beta\right], \\
    \EE(Z_i^m\log^2 Z_i) &= \beta^m \left[\alpha_i\dots(\alpha_i+m-1)\right] \left\{\psi_1(\alpha_i+m) + \left[\psi(\alpha_i+m)+\log\beta\right]^2\right\}.
    \end{align*}
    }
\end{lemma}
\begin{lemma} \label{appendix-lemma:mgamma-moments-V}
    Let $\mathbf{X}=(X_1, \dots, X_k)\sim\mathcal{MG}_k(\boldsymbol{a}, \beta)$ and $\mathbf{Z}= \Delta\mathbf{X}$. Then, for $i \in [k], m, \in \mathbb{N}$:
    {\small
    \begin{align*}
    \VV(Z_i) &= \alpha_i\beta^2, \\
    \VV(Z^2_i) &= 2\alpha_i(\alpha_i+1)(2\alpha_i+3)\beta^4, \\
    \VV(\log Z_i) &= \psi_1(\alpha_i), \\
    \VV(Z_i \log Z_i) &= \alpha_i (\alpha_i+1)\beta^2 \left\{\psi_1(\alpha_i+2) + \left[\psi(\alpha_i+2)+\log\beta\right]^2\right\} - \alpha_i^2\beta^2 \left[\psi(\alpha_i+1)+\log\beta\right]^2, \\
    \CC(Z_i, Z^2_i) &= 2\alpha_i(\alpha_i+1)\beta^3, \\
    \CC(Z_i, \log Z_i) &= \beta, \\
    \CC(Z_i, Z_i \log Z_i) &= \alpha_i(\alpha_i+1)\beta^2 \left[\psi(\alpha_i+2)+\log\beta\right]^2 - \alpha_i^2\beta^2 \left[\psi(\alpha_i+1)+\log\beta\right], \\
    \CC(\log Z_i, Z_i \log Z_i) & = \alpha_i \beta \left\{\psi_1(\alpha_i+1) + \left[\psi(\alpha_i+1)+\log\beta\right]^2\right\} - \alpha_i \beta \left[\psi(\alpha_i)+\log\beta\right] \left[\psi(\alpha_i+1)+\log\beta\right].
    \end{align*}
    }
\end{lemma}
\end{document}